\documentclass[12pt,fullpage]{article}

\usepackage{amsmath,amsthm,amsfonts,amssymb}
\usepackage[scr=boondoxupr]{mathalfa}
\usepackage{array}
\usepackage{fullpage}
\usepackage{thmtools,mathtools}
\usepackage{enumitem}
\usepackage{graphicx}
\usepackage{listings}
\usepackage{gensymb}
\usepackage{tikz}
\usetikzlibrary{decorations.pathreplacing,angles,quotes}
\usepackage{tikz-cd}
\usepackage{nameref}
\usepackage[colorlinks=true, citecolor=blue, linkcolor=teal, urlcolor=black]{hyperref}
\usepackage{multirow}
\usepackage{tikz-network}
\usepackage{caption}
\usepackage{stackengine}
\usepackage{tabularx}
\usepackage{xfrac}
\usepackage{subcaption}
\usepackage{nicematrix}
\allowdisplaybreaks
\allowbreak

\newcommand{\Q}{\mathbb{Q}}

\renewcommand{\P}{\mathbb{P}}

\renewcommand{\l}{\ell}
\renewcommand{\mod}[1]{\,({\mathrm{mod\;\,} #1})}

\newcounter{casenum}

\newcounter{casenum1}

\pgfdeclarelayer{background}
\pgfsetlayers{background,main}

\usetikzlibrary{decorations.pathreplacing,calc}
\tikzset{%
  middle dotted line/.style={
    decoration={show path construction, 
      lineto code={
          \draw[#1] (\tikzinputsegmentfirst) --($(\tikzinputsegmentfirst)!.3333!(\tikzinputsegmentlast)$);,
          \draw[dotted,#1] ($(\tikzinputsegmentfirst)!.3333!(\tikzinputsegmentlast)$)--($(\tikzinputsegmentfirst)!.6666!(\tikzinputsegmentlast)$);,
          \draw[#1] ($(\tikzinputsegmentfirst)!.6666!(\tikzinputsegmentlast)$)--(\tikzinputsegmentlast);,
      }
    },
    decorate
  },
}

\newcounter{rawrem} 
\newcounter{rawclaim} 
\newcounter{raw} 

\newcounter{dummy} \numberwithin{dummy}{section}

\newtheorem{defn}[raw]{Definition}
\newtheorem{rem}[rawrem]{Remark}

\newtheorem{lemma}[dummy]{Lemma}

\newtheorem{thrm}[dummy]{Theorem}

\newtheorem{propn}[dummy]{Proposition}

\theoremstyle{definition}

\makeatletter
\newenvironment{bproof}[1][\proofname] {\par\pushQED{\qed}\normalfont\topsep6\p@\@plus6\p@\relax\trivlist\item[\hskip\labelsep\bfseries#1\@addpunct{.}]\ignorespaces}{\popQED\endtrivlist\@endpefalse}
\makeatother

\newcommand{\emailad}[1]{\textit{Email address: }\texttt{\href{mailto:#1}{#1}}}

\newcommand{\rbx}{\mathrm{\mathbf{x}}}
\newcommand{\trunk}{T^{\circ}}

\DeclareMathOperator{\diam}{diam}
\DeclareMathOperator{\twig}{\textsc{Twig}}

\newcommand{\compalpha}[1]{\alpha_1\,\alpha_2\,\cdots\,\alpha_{#1}}
\newcommand{\compbeta}[1]{\beta_1\,\beta_2\,\cdots\,\beta_{#1}}
\newcommand{\U}{\mathit{U}}

\newcommand{\rbX}{\mathrm{\mathbf{X}}}
\newcommand{\csf}[1]{\operatorname{\rbX}_{#1}}
\newcommand{\Sym}[1]{\operatorname{Sym}_{#1}(\rbx)}

\newcommand{\Kappa}{\scalebox{1.1}{$\kappa$}}
\newcommand{\ps}{\scalebox{1.2}{$p$}}

\newcommand{\mult}{\!\cdot\!}
\renewcommand{\thefootnote}{\textit{\alph{footnote}}}
\makeatletter

\renewcommand{\thefootnote}{\textit{\alph{footnote}}}
\renewcommand{\thanks}[1]{\textbf{Please don't use ``thanks''.}}

\def\get@fnmark#1{%
  \begingroup
     \c@footnote #1\relax
     \unrestored@protected@xdef\@thisfnmark{\thefootnote}%
  \endgroup
  \@thisfnmark}
\def\authornote#1{\unskip\kern0.03em\proc@anote#1,,,,\endproc@anote}
\def\proc@anote#1,#2,#3,#4,#5\endproc@anote
{\def\@anoteA{\get@fnmark{#1}}%
  \ifx\relax#2\relax\let\@anoteB=\relax\else\def\@anoteB{,\get@fnmark{#2}}\fi
  \ifx\relax#3\relax\let\@anoteC=\relax\else\def\@anoteC{,\get@fnmark{#3}}\fi
  \ifx\relax#4\relax\let\@anoteD=\relax\else\def\@anoteD{,\get@fnmark{#4}}\fi
  \rlap{\@textsuperscript{\normalfont\@anoteA\@anoteB\@anoteC\@anoteD}}}
\long\def\the@authortext{}
\long\def\authortext#1#2{%
  \ifx\relax#1\relax
  \g@addto@macro\the@authortext{\footnotetext[0]{\raggedright #2}}\else
  \ifx#1 % intentional space before percent
  \g@addto@macro\the@authortext{\footnotetext[0]{\raggedright #2}}\else
  \g@addto@macro\the@authortext{\footnotetext[#1]{\raggedright #2}}\fi\fi}

\newcommand\email[1]{\texttt{#1}}

\long\def\@makefntext#1{\parindent 1.02em\noindent
   \everypar={\hangindent=1.02em}\hangindent=1.02em
  \hb@xt@1em{\hss\@textsuperscript{\normalfont\@thefnmark}}\kern0.02em #1}

\renewcommand\maketitle{\par
  \begingroup
    \def\@makefnmark{\rlap{\@textsuperscript{\normalfont\@thefnmark}}}%
    \newpage
    \global\@topnum\z@   % Prevents figures from going at top of page.
    \the@authortext
    \@maketitle
    \thispagestyle{plain} % \@thanks
  \endgroup
  \setcounter{footnote}{0}%
  \renewcommand{\thefootnote}{\arabic{footnote}}
  \global\let\maketitle\relax
  \global\let\@maketitle\relax
  \global\let\@author\@empty
  \global\let\@date\@empty
  \global\let\@title\@empty
  \global\let\title\relax
  \global\let\author\relax
  \global\let\date\relax
  \global\let\and\relax}
\makeatother

\title{Proper $q$-caterpillars are distinguished by their Chromatic Symmetric Functions}
\author{G. Arunkumar\authornote{1}
\and Narayanan Narayanan\authornote{1}
\and  Raghavendra Rao B. V.\authornote{2}
\and Sagar S. Sawant\authornote{1}}

\authortext{1}{Department of Mathematics, IIT Madras, Chennai 600036 (\email{garunkumar@iitm.ac.in}, \email{naru@iitm.ac.in}, \email{sagar@smail.iitm.ac.in}).}

\authortext{2}{Department of Computer Science and Engineering, IIT Madras, Chennai 600036 (\email{bvrr@iitm.ac.in)}.}
\date{}

\begin{document}

\maketitle

\begin{abstract}
    Stanley's Tree Isomorphism  Conjecture posits that the chromatic symmetric function can distinguish non-isomorphic trees. This conjecture is already established for caterpillars and other subclasses of trees. We prove the conjecture's validity for a new class of trees that generalize proper caterpillars, thus confirming the conjecture for a broader class of trees.
    % Additionally, we exhibit a new multiplication operation on the symmetric functions such that the Tutte symmetric function of join of graphs splits into the respective Tutte symmetric functions of the individual graphs. This finding sheds new light on the interplay between graph operations and symmetric functions. 

\end{abstract}
\noindent\textbf{Mathematics Subject Classification$(2010)$: } 05C15, 05C25, 05C31, 05C60.\\[1ex]\textbf{Keywords: }chromatic symmetric function, $U$-polynomial, integer compositions, caterpillars.

\section{Introduction}

The chromatic symmetric function, introduced by R. Stanley~\cite{Stanley1995}, generalizes the chromatic polynomial of a graph to a symmetric function and is widely studied as it determines  generating functions of various graph statistics. This raises the question: Does the chromatic symmetric function distinguish all graphs up to isomorphism? Unfortunately, the answer is negative. Stanley presented two non-isomorphic graphs (both containing cycles) that share the same chromatic symmetric function. However, the question remains open for trees and is conjectured to be true, famously known as Stanley's Tree Isomorphism  Conjecture. Substantial progress has been made in confirming the conjecture for various subclasses of trees. Martin et al. \cite{Martin2008OnDT} proved its validity for specific classes of caterpillars, spiders, and certain unicyclic graphs. Additionally, Aliste-Prieto and Zamora \cite{AlistePrieto2014} showed that the conjecture holds for proper caterpillars, while Loebl and Sereni \cite{loebl} extended this result to all caterpillars. To explore additional examples of graphs that can be distinguished based on their chromatic symmetric function, refer to \cite{AlistePrieto2022, crewnote,  spiders-and-their-kin, huryn,Tsujie-CSF}.

Let $\Q$ and $\P$ denote the rational numbers and positive integers, respectively. In this paper, we consider the following generalization of proper caterpillars.
\begin{defn}[proper $q$-caterpillars] 
	Let $q\geq1$ be fixed. A proper $q$-caterpillar $T$ is constructed as follows: We begin with a path $S = \langle v_1,\ldots, v_\ell \rangle$ (with endpoints $v_1$ and $v_{\l}$)  called the spine, with $\ell >0$. For every $1 \le i \le \ell$, we glue (endpoint of the path identified with a vertex on the spine) $p_i$ additional paths of length exactly $q$ to the vertices $v_i$, respectively, where $p_i \in \P$.
\end{defn}
 In this context, proper $1$-caterpillars have been distinguished by their chromatic symmetric functions up to isomorphism~\cite{AlistePrieto2014}. We prove that for all $q \geq 2$, the proper $q$-caterpillars are distinguished by their chromatic symmetric function.

\begin{thrm}{\label{thm:proper-p-caterpillars}}
    For $q>1$, the chromatic symmetric function distinguishes isomorphism classes of proper $q$-caterpillars.
\end{thrm}

 \noindent The proof uses ideas involved in \cite{AlistePrieto2014}, that is, associating proper $q$-caterpillars with the integer compositions, and the interrelations of the chromatic symmetric function, $U$-polynomial and $\mathcal{L}$-polynomial. The $\U$-polynomial, introduced by Noble and Welsh \cite{noblewelsh}, is a Tutte-Grothendieck invariant equivalent to the chromatic symmetric function when restricted to trees; that is, one can be obtained from the other by certain change of variables. Consequently, Stanley's Tree Isomorphism  Conjecture is equivalent to distinguishing trees by their $\U$-polynomial.

 Note that for $q \geq 2$, every integer composition $(p_1,p_2,\dots,p_{\l})$ with each component being positive corresponds to a unique proper $q$-caterpillar with $p_i$ number of length $q$ paths incident to the vertex $v_i$ on the spine $\langle v_1,v_2,\dots,v_{\l} \rangle$. Therefore Theorem \ref{thm:proper-p-caterpillars} states that for each such integer composition, there are infinitely many trees (one for each $q\geq 2$) that can be distinguished by chromatic symmetric function, thereby attaining a significant improvement in the pool of trees that are known to satisfy Stanley's Tree Isomorphism  Conjecture.

In Section \ref{subsec:intro}, we introduce fundamental graph notions and preliminaries, including the chromatic symmetric function, Tutte symmetric function, and $U$-polynomial. We explore the interrelations between these functions. In Section \ref{subsec:characterize-proper-cat}, we present a characterization of the proper $q$-caterpillar based on the tree statistics encoded in the chromatic symmetric function. Additionally, we define the $\mathcal{L}$-polynomial associated with integer compositions, which plays a crucial role in determining the isomorphism classes of proper $q$-caterpillars. The factorization of integer compositions is discussed in Section \ref{subsec:monoidcomp}, which specifies integer compositions with the same $\mathcal{L}$-polynomial. Section \ref{subsec:proofcat} establishes a correspondence between proper $q$-caterpillars and integer compositions, where the $U$-polynomial of the caterpillars determines the $\mathcal{L}$-polynomial of the corresponding composition. This, combined with the factorization in Section \ref{subsec:monoidcomp}, leads to the proof of Theorem \ref{thm:proper-p-caterpillars}. 
% Section \ref{sec:tsfjoin} is dedicated to studying the Tutte symmetric function and its behavior concerning the join operation. We introduce a new multiplication operation and explore its consequences in relation to the chromatic symmetric function.

%---------------------------------------------------------------------

\subsection{Preliminaries}{\label{subsec:intro}}

Let $G = (V, E)$ be a simple graph with the set of vertices $V$ and the set of edges $E$. A \emph{$\P$-coloring} of a graph $G$ is a function $f: V \rightarrow \P$, and such a coloring is said to be \emph{proper} if for every edge $uv \in E$, the colors $f(u)$ and $f(v)$ are distinct. The \emph{content of a coloring} $f$ is the $\P$-tuple $\left(\vert f^{-1}(1) \vert,\vert f^{-1}(2) \vert,\vert f^{-1}(3) \vert,\dots\right)$ denoted by $c(f)$, that encodes the cardinality of the color classes of $f$. Throughout this paper, whenever we consider a $\P$-tuple, we assume that its all but finitely many components are zero. In what follows, we adopt the graph notions and terminology in accordance with~\cite{bondy2011graph}.

Let $\rbx= (x_1,x_2,\dots)$ be the collection of commutative indeterminates. For a $\P$-tuple $\alpha$, let $\rbx^{\alpha}$ be the monomial having the $i^{\mathrm
{th}}$ component of $\alpha$ as the exponent of $x_i$. A \emph{symmetric function} is a formal power series in indeterminates $\rbx$ that is invariant under all permutations of $\rbx$. Let $\Sym{R}$ denote the collection of symmetric functions with coefficients over ring $R$. We now define the two symmetric functions introduced by Stanley~\cite{Stanley1998, Stanley1995}.

The \emph{chromatic symmetric function} of a graph $G = (V,E)$ is defined as
\begin{equation}
    \csf{G} := \sum_{\substack{f:V \rightarrow \P   \\   \textit{proper}}} \rbx^{c(f)}.
\end{equation}
% and the \emph{Tutte symmetric function is defined as 
% \begin{equation}
%     \tsf{G} :=  \sum_{f:V \rightarrow \P} \rbx^{c(f)}(1+y)^{\vert f^=\vert},
% \end{equation}}
% where $f^=$ is the set of edges with monochromatic endpoints under $f$. We call edges in $f^=$ as monochromatic edges under $f$. 

The above function is indeed symmetric in $\rbx$ since the permutation of the colors does not affect the properness of colorings. Also, the functions are homogeneous in $\rbx$ with degree $|V|$. 
% Note that the Tutte symmetric function of a graph is mapped to its chromatic symmetric function under the substitution homomorphism $\Epsilon: \Sym{\Q[y]} \rightarrow \Sym{\Q}$ mapping $f(\rbx;y)$ to $f(\rbx;-1)$.

The coefficients arising in the expansion of the chromatic symmetric function in various basis of the $\Sym{\Q}$ encodes numerous combinatorics of the graph. We are particularly interested in the expansion with respect to the power sum symmetric function basis. 

For $k \in \P$, the \emph{power sum symmetric function of degree $k$} is defined as 
\begin{equation*}
    \ps_k(\rbx) = \sum_{i \in \P} x_i^k.
\end{equation*}

An integer \emph{partition} of a positive integer $n$ is a weakly decreasing sequence of positive integers whose sum is $n$. For any partition $\lambda=\lambda_1 \, \lambda_2 \, \cdots \,\lambda_k \vdash n$, the power sum symmetric function
\begin{equation*}
    \ps_{\lambda}(\rbx) = \prod_{i=1}^k \ps_{\lambda_i}(\rbx),
\end{equation*}
and the collection $\left\{\ps_{\lambda}\right\}_{n,\lambda\vdash n}$ forms a $\Q$-basis of $\Sym{\Q}$. 

 Given a graph $G=(V,E)$ and a subset $F \subseteq E$, let $\lambda[F]$ be the partition of $|V|$ formed by the orders of the connected components of the spanning subgraph $G[F]$. 
\begin{thrm}[{\cite[Theorem 2.5]{Stanley1995}}]
    For a graph $G$, the expansion of the chromatic symmetric function in the power sum symmetric function basis is
    \begin{equation}{\label{eq:csfpowersum-expansion}}
        \csf{G} = \sum_{F \subseteq E} (-1)^{|F|}\ps_{\lambda[F]}(\rbx).
    \end{equation}
\end{thrm}

The expansion \eqref{eq:csfpowersum-expansion} of the chromatic symmetric functions establishes a strong connection with the $\U$-polynomial defined by Noble and Welsh~\cite{noblewelsh}. For a partition $\lambda=\lambda_1 \, \lambda_2 \, \cdots \,\lambda_k \vdash n$, let $\rbx_{\lambda}$ denote the monomial $x_{\lambda_1}x_{\lambda_2}\cdots x_{\lambda_k}$, 

\begin{defn}
    Given a graph $G=(V,E)$, the $\U$-polynomial of the graph is defined as 
	\begin{equation*}
		\U_G(\rbx;y) = \sum_{F \subseteq E} \rbx_{\lambda[F]}(y-1)^{|F|-|V|+\Kappa(F)},
	\end{equation*}
where $\Kappa(F)$ is the number of connected components in the spanning subgraph $G[F]$, or equivalently the length of partition $\lambda[F]$. 
\end{defn}

Due to the fact that any spanning subgraph $T[F]$ of a tree $T=(V,E)$ must have $|V|-|F|$ connected components, $|F|-|V| + \Kappa(F) = 0$ for all $F \subseteq E$. Therefore, for any tree $T$, we have 
\begin{equation}{\label{eq:chromatic-U-equivalence}}
    (-1)^{|V|}\U_T\left({-}p_1(\rbx),{-}p_2(\rbx),{-}p_3(\rbx),\dots;y\right) = (-1)^{|V|}\sum_{F \subseteq E} (-1)^{\Kappa(F)}\ps_{\lambda[F]}(\rbx)=\csf{T}.
\end{equation}
This implies that the two graph invariants are equivalent when restricted to trees, that is, two trees have the same chromatic symmetric function if and only if they have the same $\U$-polynomial.

%----------------------------------------------------------------------------------
%----------------------------------------------------------------------------------

\section{Proper \emph{q}-Caterpillars}{\label{sec:proper-cat}}

% This section focuses on distinguishing isomorphism classes of proper $q$-caterpillars by chromatic symmetric function. 
We prove Theorem~\ref{thm:proper-p-caterpillars} in this section. 
We begin by characterizing proper $q$-caterpillars in terms of the statistics that can be identified through the chromatic symmetric function. This allows us to differentiate proper $q$-caterpillars from other types of trees.

\subsection{Characterization of proper \emph{q}-caterpillars}{\label{subsec:characterize-proper-cat}}
Given a tree $T=(V,E)$, the {\em trunk} $T^{\circ}$ of $T$ is the smallest subtree containing all vertices of degree at least three. For each pendant vertex $u$ of $T$, there exists a unique path starting at $u$ and ending at some vertex in the trunk such that all internal vertices of the path have degree two. Each such path is called a \emph{twig}, and let $\twig(T)$ be the multiset representing the lengths of twigs in $T$. Evidently, every tree containing a vertex of degree at least three can be decomposed into the trunk $\trunk$ and some twigs. Crew proved that the order of $\trunk$, and the multiset $\twig(T)$ can be determined by the chromatic symmetric function~\cite{crewnote}.

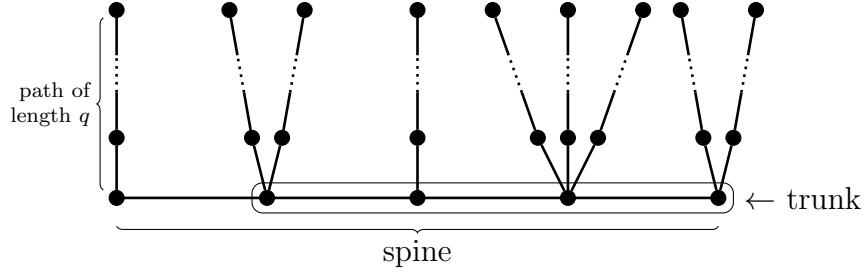
\begin{figure}
    \centering
    \begin{tikzpicture}[vertex/.style={draw,circle,fill=black,inner sep=2pt}]
    %drawing the spine
        \foreach \x in {0,1,2,3,4}
            {
            \node[vertex] (v\x) at (2*\x,0){};
            }
        \foreach \x in {0,1,2,3}
            {
            \draw[black,line width=1pt] (2*\x,0) -- (2*\x+2,0);
            }

    %paths at v0        
        \node[vertex] (u11) at (0,0.8) {};    
        \node[vertex] (u12) at (0,2.5) {};
        \draw[black,line width=1pt] (v0) -- (u11);
    %paths at v1
        \node[vertex,shift={(2,0)}] (u21) at (-0.2,0.8) {};
        \node[vertex,shift={(2,0)}] (u31) at (0.2,0.8) {};            
        \node[vertex,shift={(2,0)}] (u22) at (-0.5,2.5) {};
        \node[vertex,shift={(2,0)}] (u32) at (0.5,2.5) {};
        \foreach \x in {2,3} {\draw[black,line width=1pt] (v1) -- (u\x1) ;}
    %paths at v2
        \node[vertex,shift={(4,0)}] (u41) at (0,0.8) {};   \node[vertex,shift={(4,0)}] (u42) at (0,2.5) {};
        \draw[black,line width=1pt] (v2) -- (u41); 
    %paths at v3
        \node[vertex,shift={(6,0)}] (u51) at (-0.4,0.8) {};
        \node[vertex,shift={(6,0)}] (u61) at (0,0.8) {}; 
        \node[vertex,shift={(6,0)}] (u71) at (0.4,0.8) {}; 
        \node[vertex,shift={(6,0)}] (u52) at (-1,2.5) {};
        \node[vertex,shift={(6,0)}] (u62) at (0,2.5) {};
        \node[vertex,shift={(6,0)}] (u72) at (1,2.5) {};
        \foreach \x in {5,6,7} {\draw[black,line width=1pt] (v3) -- (u\x1) ;}    
    %paths at v4
        \node[vertex,shift={(8,0)}] (u81) at (-0.2,0.8) {};    
        \node[vertex,shift={(8,0)}] (u91) at (0.2,0.8) {};
        \node[vertex,shift={(8,0)}] (u82) at (-0.5,2.5) {};
        \node[vertex,shift={(8,0)}] (u92) at (0.5,2.5) {};
        \foreach \x in {8,9} {\draw[black,line width=1pt] (v4) -- (u\x1) ;}    
    %dotted paths to the leaves        
        \foreach \x in {1,2,3,4,5,6,7,8,9}
        {
        \draw[middle dotted line={line width=1pt}] (u\x1) -- (u\x2);
        }
    % trunk
        \draw[rounded corners,thin] (1.8,-0.2) rectangle (8.2,0.2) {};
    % brace
        \draw[decoration={brace,raise=5pt},decorate]
        (v0) -- node[left=6pt,align=right] {\scriptsize path of \\[-5pt]  \scriptsize{length $q$}} (u12);
        \draw[decoration={brace,mirror,raise=5pt},decorate]
        (0,-0.2) -- node[below=6pt] {spine} (8,-0.2);
        
    % text
        \node[anchor=west] at (8.2,0) {$\leftarrow$ trunk};
    \end{tikzpicture}
    \caption{Example of a proper $q$-caterpillar with spine of order $5$, trunk of order $4$ and the multiset of twigs is $\{\protect\underbrace{q,q,\dots,q}_{\text{8 times}},q{+}1 \}$.}
    \label{fig:cattrunk}
\end{figure}

It is clear that a path $T$ is a proper $q$-caterpillar if and only if its order is either $q+1,2q+1$ or $2q+2$. For proper $q$-caterpillars that are not paths, we have the following characterization in terms of the tree-invariants that can be obtained from the chromatic symmetric function, such as, order of the tree, degree sequence, multiset of twigs and the diameter of the tree.
\begin{propn}{\label{prop: csf-q-caterpillars}}
	Let $q \geq 1$ be fixed and $T=(V,E)$ be a tree that is not a path. Then $T$ is a proper $q$-caterpillar if and only if it satisfies the following:
	\begin{enumerate}[label=(\roman*)]
		\item $|\trunk| = |V| - \delta_1 - \delta_2$ where $\delta_i$ is the number of vertices of degree $i$ in $T$.
		\item $\twig(T)$ only contains integers $q$ and $q{+}1$, with $m_{q{+}1} \leq 2$ where $m_{q{+}1}$ is the multiplicity of $q{+}1$ in $\twig(T)$.
		\item $\diam(T) = (|\trunk|-1)+2q+m_{q{+}1}$.
	\end{enumerate}
\end{propn}
\begin{proof}
$(\Rightarrow)$ It is clear that every proper $q$-caterpillar that is not a path satisfies the above three conditions. \\ 
$(\Leftarrow)$ A tree satisfying $|\trunk|=1$ and $(ii)$ is indeed a proper $q$-caterpillar. Thus we may assume that $|\trunk| \geq 2$. Note that $\diam(T) \leq 2q+\diam(\trunk) + m_{q+1}  $ along with $(iii)$ implies that $(|\trunk|-1)\leq \diam(\trunk)  $, and hence $\trunk$ is a path, say $\langle w_1,w_2,\dots,w_{k} \rangle$ (with endpoints $w_1$ and $w_{k})$. From $(i)$, it follows that $\trunk$ consists only of vertices of degree at least $3$, owing to which every vertex of the trunk must be incident to at least one twig. To prove that $T$ is a proper $q$-caterpillar, it suffices to prove that twigs of length $q+1$ (if they exist) are incident to the distinct endpoints of the trunk. For $1 \le i \le k$, let $w_i$ be incident to $n_i$ many twigs $P_i^t$ ($1 \le t \le n_i$), and fix $0 \le r_i \le n_i$ where the length of the path $P_i^t$ is equal to $q$ if $r_i < t \le n_i$ and $q+1$ otherwise. Let $u_i^t$ be the pendant vertex of the twig $P_i^t$($1 \le t \le n_i$). In the resulting tree $T$, we have the following

    $$d(u_i^t, u_j^s) = \begin{cases}
        q + |i-j| + q 
        &\text{ if } r_i<t\le n_i \text{ and } r_j < s \le n_j,\\
        (q+1) + |i-j| + q 
        &\text{ if } 1 \le t \le r_i \text{ and } r_j < s \le n_j, \\
        q + |i-j| + (q+1)
        &\text{ if } r_i < t \le n_i \text{ and } 1 \le s \le r_j, \\
        (q+1) + |i-j| + (q+1)
        &\text{ if } 1 \le t \le r_i \text{ and } 1 \le s \le r_j.
    \end{cases}$$
From the above computation, the endpoints of the path in $T$ of length $\diam(T)$ must be $w_{1}^t$ and $w_{k}^s$ for some $1 \leq t \leq n_1$ and $1 \leq s \leq n_{k}$. This together with $(ii)$ and $(iii)$ dictates the position of $q{+}1$-twigs as follows:
$$m_{q+1} = \begin{cases}
    0 & \text{ iff } r_1 = r_k = 0, \\
    1 &\text{ iff exactly one of $r_1$ or $r_k$ is non-zero,}  \\
    2 &\text{ iff both $r_1$ and $r_k$ are non-zero.} 
\end{cases} $$
Therefore, the tree $T$ is a proper $q$-caterpillar, and this completes the proof. 
\end{proof}

\noindent\textbf{Note. }The trunk of the proper $q$-caterpillar may not coincide with the spine (see Figure \ref{fig:cattrunk}). However, it is always a subpath of the spine. 

Proposition \ref{prop: csf-q-caterpillars} yields that the chromatic symmetric function can distinguish proper $q$-caterpillars from other trees, as all the tree-invariants used in their characterization can be extracted from the chromatic symmetric function. 

Before proceeding to the proof of Theorem \ref{thm:proper-p-caterpillars}, we revisit the factorization of integer compositions introduced in \cite{thomas-composition}. This factorization is instrumental for determining the isomorphism classes of proper $q$-caterpillars.

%----------------------------------------------------------------------------------------

\subsection{Monoid of Integer Compositions}{\label{subsec:monoidcomp}}

Let $n \in \P$ be a positive integer. An \emph{integer composition} $\alpha$ of $n$, denoted by  $\alpha \vDash n$, is an ordered sequence of positive integers $\alpha_1\,\alpha_2\,\cdots\,\alpha_r$ whose sum is $n$. We call each $\alpha_i$ as the $i^{\mathrm{th}}$ component of $\alpha$, and its length $\l(\alpha)$ is $r$. Let $\mathscr{C}$ denote the set of all integer compositions.

For any two compositions $\alpha = \alpha_1\,\alpha_2\,\cdots\,\alpha_{r}$ and $\beta = \compbeta{s}$, their \emph{concatenation} is given by
$$\alpha \cdot \beta := \compalpha{r}\;\compbeta{s},$$
whereas the near-concatenation operation is defined as
$$ \alpha \odot \beta := \alpha_1\,\alpha_2\,\cdots\,\alpha_{r-1}\, (\alpha_{r}+\beta_1) \, \beta_2 \, \cdots \,\beta_{s}.$$ Let $\alpha^{\odot q}$ denote $q$-fold near-concatenation $\underbrace{\alpha  \odot \alpha \odot \cdots \odot \alpha}_{q \text{ times}}$, for any positive integer $q$.
The \emph{composition} of two integer compositions is given by 
$$
\alpha \circ \beta := \beta^{\odot\alpha_1}\, \cdot\,\beta^{\odot\alpha_2}\cdot\,\cdots\,\cdot \beta^{\odot{\alpha_{r}}}.
$$
For example, $2\,1 \circ 2\,3 = (2\,3\,\odot 2\,3) \cdot 2\,3 = 2\,5\,3\,2\,3$.
\begin{propn}[{\cite[Proposition~3.3]{thomas-composition}}] $\left( \mathscr{C}, \circ \right)$ is a non-commutative monoid with the integer composition $1$ as the identity element.    
\end{propn}

Let $\alpha$ and $\beta$ be two integer compositions of $n$. Then $\alpha$ is said to be a \emph{refinement} of $\beta$ if $\alpha$ is obtained by subdividing some (or no) parts of $\beta$, denoted by $\alpha \preceq \beta $. In this case, we also say that $\beta$ is a \emph{coarsening} of $\alpha$. For example, $2\,3\,1\,3\,2 \preceq 2\,4\,5$. Let $\left(\mathscr{C},\preceq \right)$ be the poset with the refinement order. In \cite{thomas-composition}, Billera, Thomas and Willigenburg defined an equivalence relation on $\mathscr{C}$ based on the refinement of the integer compositions. We consider the polynomial interpretation of that equivalence relation called the $\mathcal{L}$-polynomial \cite{AlistePrieto2014, thomas-composition}. 

The $\mathcal{L}$-polynomial of an integer composition $\alpha$ is defined as
\begin{equation*} 
    \mathcal{L}(\rbx;\alpha) = \sum_{\beta \succeq \alpha} x_{\beta_1}x_{\beta_2}\dots x_{\beta_{r}}.
\end{equation*}
For instance, the $\mathcal{L}$-polynomial of the composition $2\,2\,1\,2$ is $x_1x_2^3 + x_1x_2x_4 + 2x_2^2x_3+ 2x_2x_5 +x_3x_4+x_7$.
Note that the equality of the $\mathcal{L}$-polynomial induces an equivalence relation on the integer compositions. Let $[\alpha]_{\mathcal{L}}$ denote the equivalence class of $\alpha$ under this equivalence relation. We recall its description using the unique factorization in $\left( \mathscr{C}, \circ \right)$~\cite{thomas-composition}.

A factorization $\alpha = \varepsilon \circ \eta$ is said to be \emph{trivial} if one of the following is satisfied:
\begin{enumerate}[label=\alph*)]
    \item either $\varepsilon$ or $\eta$ is an identity composition $1$,
    \item both $\varepsilon$ and $\eta$ are of length $1$,
    \item both $\varepsilon$ and $\eta$ have all parts equal to $1$.
\end{enumerate}

An integer composition is said to be \emph{irreducible} if it admits only trivial factorization. A factorization $\alpha = \eta_1 \circ \eta_2 \circ \cdots \circ \eta_{k}$ is said to be an \emph{irreducible factorization} if each integer composition $\eta_i$ is irreducible and no $\eta_i\circ \eta_{i+1}$ is a trivial factorization.

\begin{thrm}[{\cite[Theorem 3.6]{thomas-composition}}]
    Every integer composition admits a unique irreducible factorization.
\end{thrm}
For an integer composition $\alpha$, let $\alpha^{*}$ be the integer composition obtained by reversing $\alpha$, that is, the $i^{\mathrm{th}}$ component of $\alpha^{*}$ is $\alpha_{\l(\alpha)-i+1}$ for $1 \leq i \leq \l(\alpha)$.
\begin{thrm}[{\cite[Theorem 4.1]{thomas-composition}}]{\label{thm:factorize}}
    Let $\alpha = \eta_1 \circ \eta_2 \circ \cdots \circ \eta_{k}$ be the irreducible factorization of $\alpha$. Then
    \begin{equation*}
        [\alpha]_{\mathcal{L}} = \left\{\varepsilon_1 \circ \varepsilon_2 \circ \cdots \circ \varepsilon_{k} ~ \mid ~  \varepsilon_i = \eta_i \text{ or } \varepsilon_i={\eta_i}^*,\text{ for all }i=1,2,\dots k \right\},
    \end{equation*}
\end{thrm}
\noindent \textbf{Example. } Consider the integer composition $4\;10\;4\;10 $ with its irreducible factorization given by $1\;1 \circ 2\;5 \circ 2$. Then the equivalence class 
$$
[4\;10\;4\;10]_{\mathcal{L}} = \{1\;1 \circ 2\;5 \circ 2 , 1\;1 \circ 5\;2 \circ 2\} = \{4\;10\;4\;10,10\;4\;10\;4\}.
$$
%---------------------------------------------------------------------------------
\subsection{Proof of Theorem \ref{thm:proper-p-caterpillars}}{\label{subsec:proofcat}}

We have already established in Section \ref{subsec:characterize-proper-cat} that the chromatic symmetric function determines proper $q$-caterpillars. Our objective now is to demonstrate that the chromatic symmetric function distinguishes non-isomorphic proper $q$-caterpillars from one another. To accomplish this, we associate every proper $q$-caterpillar with a unique integer composition such that any two proper $q$-caterpillars are isomorphic if and only if their corresponding compositions are either the same or reverses of one another. The proof technique is similar to the proof of distinguishing proper $1$-caterpillars \cite{AlistePrieto2014}.

Let $q \in \{2,3,4,\dots\}$, and $T$ be a proper $q$-caterpillar. Let $\langle v_1, v_2, \dots v_{\l} \rangle$ denote the spine of $T$. Let $p_i$ represent the number of paths in $T$ of length $q$, starting from a leaf and ending at $v_i$. We define a composition $\varphi(T)$ of length $\l$ whose $i^{\mathrm{th}}$ component is $q {\cdot} p_i +1$.  Conversely, for any integer composition $\alpha$ with all components greater than one and congruent to $1$ modulo $q$, we construct a proper $q$-caterpillar $\tau(\alpha)$ as follows: consider a path with $\l(\alpha)$ vertices, which serves as the spine, and glue $\tfrac{\alpha_i-1}{q}$ new paths of length $q$ to the $i^{\mathrm{th}}$ vertex of the spine. The mapping $\varphi$ and $\tau$  are inverses of each other. For instance, the proper $q$-caterpillar in Figure \ref{fig:cattrunk} corresponds to the integer composition $q{+}1\;\;2q{+}1\;\;q{+}1\;\;3q{+}1\;\;2q{+}1$.
\begin{rem}{\label{rem:comp-cat}}
    Any two proper $q$-caterpillars $S$ and $T$ are isomorphic if and only if $\varphi(S) = \varphi(T)$ or $\varphi(S) = \varphi(T)^*$.
\end{rem} 

The following lemma, which is a generalization of \cite[Proposition 2.5]{AlistePrieto2014}, states that $\mathcal{L}$-polynomial of the compositions associated to proper $q$-caterpillars can be obtained as an evaluation of the $\U$-polynomial. 

\begin{lemma}{\label{lem:u-poly-R-poly equiv}}
	Let $q\geq 1$. For any proper $q$-caterpillar $T=(V,E)$ and the composition $\varphi(T)$ associated to $T$, we have 
	$$\U_T(\underbrace{0,0,\dots,0}_{q \text{ times} },x_{q+1},x_{q+2},\dots) = \mathcal{L}(\varphi(T);\rbx).$$
\end{lemma}
\begin{proof}
	The $\U$-polynomial with $x_1=x_2=\cdots=x_q=0$ can be interpreted as the subset-sum over $F \subseteq E$ such that each connected component of the induced subgraph $T[F]$ has order at least $q+1$.
	This implies that such an $F$ must contain all non-spine edges (otherwise, the induced subgraph $T[F]$ would contain a connected component of order at most $q$). Thus every monomial in $\U_T(0,0,\dots,0,x_{q+1},x_{q+2},\dots)$ corresponds to a unique subset of the spine-edges of $T$ (see Table \ref{tab:subgraph--monomials}). These subsets uniquely determine the coarsenings of the composition $\varphi(T)$ in the poset $\left(\mathscr{C},\preceq\right)$. Therefore we have 
	$$\U_T(0,0,\dots,0,x_{q+1},x_{q+2},\dots) = \sum_{\substack{F \subseteq E \\
	F\text{ contains all}\\ \text{non-spine edges}}} \rbx_{\lambda[F]} = \mathcal{L}(\varphi(T);\rbx).$$
\end{proof}
\begin{table}[]
    \centering
    \begin{NiceTabular}{|m{0.12\linewidth}|>{\centering\arraybackslash}m{0.15\linewidth}|>{\centering\arraybackslash}m{0.15\textwidth}|>{\centering\arraybackslash}m{0.15\linewidth}|>{\centering\arraybackslash}m{0.15\linewidth}|}
    \hline\\%
    Subgraphs
    &
    \begin{tikzpicture}[scale =0.5,vertex/.style={draw,circle,fill=black,inner sep=2pt, scale =0.5}]
    %drawing the spine
        \foreach \x in {0,1,2}
            {
            \node[vertex] (v\x) at (2*\x,0){};
            }

    %paths at v0        
        \node[vertex] (u11) at (0,0.8) {};  
        \node[vertex] (u12) at (0,1.6) {};
        \node[vertex] (u13) at (0,2.4) {};
        \draw[black,line width=0.5pt] (v0) -- (u11);
    %paths at v1
        \node[vertex,shift={(2,0)}] (u21) at (0,0.8) {}; 
        \node[vertex,shift={(2,0)}] (u22) at (0,1.6) {}; 
        \node[vertex,shift={(2,0)}] (u23) at (0,2.4) {};
        \draw[black,line width=0.5pt] (v1) -- (u21);  
    %paths at v2
        \node[vertex,shift={(4,0)}] (u31) at (-0.2,0.8) {};
        \node[vertex,shift={(4,0)}] (u32) at (-0.4,1.6) {};
        \node[vertex,shift={(4,0)}] (u41) at (0.2,0.8) {};    
        \node[vertex,shift={(4,0)}] (u42) at (0.4,1.6) {}; 
        \node[vertex,shift={(4,0)}] (u33) at (-0.6,2.4) {};
        \node[vertex,shift={(4,0)}] (u43) at (0.6,2.4) {};
        \foreach \x in {3,4} {\draw[black,line width=0.5pt] (v2) -- (u\x1) ;}
        
        \foreach \x in {1,2,3,4}
        {\draw[black,line width = 0.5pt] (u\x1) -- (u\x2) -- (u\x3);}
    
    \end{tikzpicture}
    &  
    \begin{tikzpicture}[scale =0.5,vertex/.style={draw,circle,fill=black,inner sep=2pt, scale =0.5}]
    %drawing the spine
        \foreach \x in {0,1,2}
            {
            \node[vertex] (v\x) at (2*\x,0){};
            }
        \foreach \x in {0}
            {
            \draw[black,line width=0.5pt] (2*\x,0) -- (2*\x+2,0);
            }

   %paths at v0        
        \node[vertex] (u11) at (0,0.8) {};  
        \node[vertex] (u12) at (0,1.6) {};
        \node[vertex] (u13) at (0,2.4) {};
        \draw[black,line width=0.5pt] (v0) -- (u11);
    %paths at v1
        \node[vertex,shift={(2,0)}] (u21) at (0,0.8) {}; 
        \node[vertex,shift={(2,0)}] (u22) at (0,1.6) {}; 
        \node[vertex,shift={(2,0)}] (u23) at (0,2.4) {};
        \draw[black,line width=0.5pt] (v1) -- (u21);  
    %paths at v2
        \node[vertex,shift={(4,0)}] (u31) at (-0.2,0.8) {};
        \node[vertex,shift={(4,0)}] (u32) at (-0.4,1.6) {};
        \node[vertex,shift={(4,0)}] (u41) at (0.2,0.8) {};    
        \node[vertex,shift={(4,0)}] (u42) at (0.4,1.6) {}; 
        \node[vertex,shift={(4,0)}] (u33) at (-0.6,2.4) {};
        \node[vertex,shift={(4,0)}] (u43) at (0.6,2.4) {};
        \foreach \x in {3,4} {\draw[black,line width=0.5pt] (v2) -- (u\x1) ;}
        
        \foreach \x in {1,2,3,4}
        {\draw[black,line width = 0.5pt] (u\x1) -- (u\x2) -- (u\x3);}
    \end{tikzpicture}
    &
    \begin{tikzpicture}[scale =0.5,vertex/.style={draw,circle,fill=black,inner sep=2pt, scale =0.5}]
    %drawing the spine
        \foreach \x in {0,1,2}
            {
            \node[vertex] (v\x) at (2*\x,0){};
            }
        \foreach \x in {1}
            {
            \draw[black,line width=0.5pt] (2*\x,0) -- (2*\x+2,0);
            }

    %paths at v0        
        \node[vertex] (u11) at (0,0.8) {};  
        \node[vertex] (u12) at (0,1.6) {};
        \node[vertex] (u13) at (0,2.4) {};
        \draw[black,line width=0.5pt] (v0) -- (u11);
    %paths at v1
        \node[vertex,shift={(2,0)}] (u21) at (0,0.8) {}; 
        \node[vertex,shift={(2,0)}] (u22) at (0,1.6) {}; 
        \node[vertex,shift={(2,0)}] (u23) at (0,2.4) {};
        \draw[black,line width=0.5pt] (v1) -- (u21);  
    %paths at v2
        \node[vertex,shift={(4,0)}] (u31) at (-0.2,0.8) {};
        \node[vertex,shift={(4,0)}] (u32) at (-0.4,1.6) {};
        \node[vertex,shift={(4,0)}] (u41) at (0.2,0.8) {};    
        \node[vertex,shift={(4,0)}] (u42) at (0.4,1.6) {}; 
        \node[vertex,shift={(4,0)}] (u33) at (-0.6,2.4) {};
        \node[vertex,shift={(4,0)}] (u43) at (0.6,2.4) {};
        \foreach \x in {3,4} {\draw[black,line width=0.5pt] (v2) -- (u\x1) ;}
        
        \foreach \x in {1,2,3,4}
        {\draw[black,line width = 0.5pt] (u\x1) -- (u\x2) -- (u\x3);}
    \end{tikzpicture}
    &
    \begin{tikzpicture}[scale =0.5,vertex/.style={draw,circle,fill=black,inner sep=2pt, scale =0.5}]
    %drawing the spine
        \foreach \x in {0,1,2}
            {
            \node[vertex] (v\x) at (2*\x,0){};
            }
        \foreach \x in {0,1}
            {
            \draw[black,line width=0.5pt] (2*\x,0) -- (2*\x+2,0);
            }

    %paths at v0        
        \node[vertex] (u11) at (0,0.8) {};  
        \node[vertex] (u12) at (0,1.6) {};
        \node[vertex] (u13) at (0,2.4) {};
        \draw[black,line width=0.5pt] (v0) -- (u11);
    %paths at v1
        \node[vertex,shift={(2,0)}] (u21) at (0,0.8) {}; 
        \node[vertex,shift={(2,0)}] (u22) at (0,1.6) {}; 
        \node[vertex,shift={(2,0)}] (u23) at (0,2.4) {};
        \draw[black,line width=0.5pt] (v1) -- (u21);  
    %paths at v2
        \node[vertex,shift={(4,0)}] (u31) at (-0.2,0.8) {};
        \node[vertex,shift={(4,0)}] (u32) at (-0.4,1.6) {};
        \node[vertex,shift={(4,0)}] (u41) at (0.2,0.8) {};    
        \node[vertex,shift={(4,0)}] (u42) at (0.4,1.6) {}; 
        \node[vertex,shift={(4,0)}] (u33) at (-0.6,2.4) {};
        \node[vertex,shift={(4,0)}] (u43) at (0.6,2.4) {};
        \foreach \x in {3,4} {\draw[black,line width=0.5pt] (v2) -- (u\x1) ;}
        
        \foreach \x in {1,2,3,4}
        {\draw[black,line width = 0.5pt] (u\x1) -- (u\x2) -- (u\x3);}
    \end{tikzpicture}
    \\
    \hline
    Monomials
    &
    $x_4^2x_7$
    &
    $x_7x_8$
    &
    $ x_4x_{11}$
    &
    $x_{15}$\\
    \hline
    
    \end{NiceTabular}
    \caption{Subgraphs containing all non-spine edges of a proper $3$-caterpillar, and their corresponding monomials.}
    \label{tab:subgraph--monomials}
\end{table}

The Lemma \ref{lem:u-poly-R-poly equiv} along with \eqref{eq:chromatic-U-equivalence} asserts that the chromatic symmetric function determines the $\mathcal{L}$-polynomial. We prove that the isomorphism classes of the proper $q$-caterpillars correspond to the $\mathcal{L}$-equivalence classes of integer compositions via map $\varphi$. The following lemma helps in determining the irreducible factorization of the integer compositions associated with the proper $q$-caterpillars. 
\begin{lemma}{\label{lem:compfact}}
    Let $q \geq 2$ and $h$ be positive integers such that $q$ does not divide $h$. Let $\gamma$ be an integer composition whose each component is $h$ modulo $q$ and the greatest common divisor (gcd) of all components is $1$. Then either $\gamma$ is irreducible or its irreducible factorization is $\gamma = (1^m) \circ \omega$ where $(1^m)$ denotes the integer composition of length $m$ with all components equal to $1$.
\end{lemma}
\begin{proof}
    We may assume that $\gamma$ is not irreducible. We prove using induction on length of $\gamma$. Let $\gamma = \zeta \circ \eta$ be a non-trivial factorization of $\gamma$. We claim that each component of $\zeta$ must be equal to $1$. Assume to the contrary that $\zeta$ contains at least one component greater than $1$, and let $i$ be the least index with the $i^{\mathrm{th}}$ component $\zeta_i > 1$. The gcd of all components of $\gamma$ being $1$ implies that the length of $\eta$ must be at least $2$. Since $\gamma_1 = \eta_1$ and $\gamma_{\l(\gamma)} = \eta_{\l(\eta)}$, both $\eta_1$ and $\eta_{\l(\eta)}$ are congruent to $h$ modulo $ q $. For $k=\l(\eta)\mult i$, consider the $k^{\mathrm{th}}$ component of $\gamma = \zeta \circ \eta$. By the given hypothesis, we get $\gamma_k$ to be congruent to $h$ modulo $q$, but the factorization implies     $$
   \varepsilon_k = \left(\zeta \circ \eta\right)_k = \eta_1 + \eta_{\l(\eta)} \equiv 2h \,({\mathrm{mod\;\,} q}).
    $$
   This is not possible because $h$ is non-zero modulo $q$. Therefore $\zeta$ must have all the components equal to 1, that is, $\gamma = (1^r) \circ \eta$ for some $r \geq 2$. Note that $\eta$ satisfies the given hypothesis and its length $\l(\eta) < \l(\gamma)$. Using induction, either $\eta$ is irreducible or its irreducible factorization is $(1^s) \circ \omega$, and consequently, the irreducible factorization of $\gamma$ is $(1^r) \circ \eta$ or $(1^{rs}) \circ \omega$, respectively. Thus $\gamma$ admits the required irreducible factorization.
\end{proof}

Using Lemma \ref{lem:compfact}, we can conclude that the proper $q$-caterpillars are distinguished by the chromatic symmetric functions up to isomorphism. 

\begin{bproof}[Proof of Theorem \ref{thm:proper-p-caterpillars}]
    Let $q\geq 2$. Let $S$ and $T$ be two proper $q$-caterpillars with the same chromatic symmetric function. Lemma \ref{lem:u-poly-R-poly equiv} implies that the $\mathcal{L}$-polynomial of $\varphi(S)$ and $\varphi(T)$ are equal as well. As a consequence of Remark \ref{rem:comp-cat}, it suffices to prove that the equivalence class $[\varphi(T)]_{\mathcal{L}} = \{\varphi(T),{\varphi(T)}^*\}$.
    If the gcd of all components of $\varphi(T)$ is $1$, then by Lemma \ref{lem:compfact} either $\varphi(T)$ is irreducible or its irreducible factorization is $(1^r)\circ \omega$. On the other hand, if the gcd of all components is $d$ which is greater than $1$, then factorize $\varphi(T) = \varepsilon \circ d$. Note that the gcd of all components of $\varepsilon$ is 1, and each component is congruent to $h$ modulo $q$, where $h$ is the least positive integer satisfying $d {\cdot} h \equiv 1 \mod{q}$. By Lemma \ref{lem:compfact}, either $\varepsilon$ is irreducible or its irreducible factorization must be  $(1^r) \circ \omega$ for some $r\geq 2$. This implies that the irreducible factorization of $\varphi(T)$ is $\varepsilon \circ d$ or $(1^r) \circ \omega \circ d $. In either case, the irreducible factorization of $\varphi(T)$ contains at most one non-palindrome composition. This, along with Theorem \ref{thm:factorize} concludes that $[\varphi(T)]_{\mathcal{L}} = \{\varphi(T),{\varphi(T)}^*\}$. This completes the proof.
\end{bproof}
%Note that the above proof particularly holds for $q\geq 2$ due to the condition inherited by each component of the integer compositions corresponding to a proper $q$-caterpillars. Moreover, the factorization of integer compositions used to distinguish proper $1$-caterpillars in \cite{AlistePrieto2014} are intricate, whereas for $q \geq 2$, the modular arithmetic provides a simpler factorization of the integer compositions.
%%%%%%%
%{\em We conclude by noting that the proof becomes somewhat simpler for the case $q\ge 2$ compared to the case $q=1$ proved in \cite{AlistePrieto2014} as the compositions are more restricted. 
%}
\section*{Acknowledgements} We  thank Jos\'{e} Aliste-Prieto for valuable comments and suggestions that helped improve the presentation significantly. We also thank Anna de Mier for providing valuable feedback.

%----------------------------------------------------------------------------------
%----------------------------------------------------------------------------------

%----------------------------------------------------------------------------------
%----------------------------------------------------------------------------------

%----------------------------------------------------------------------
% B I B L I O G R A P H Y 

%----------------------------------------------------------------------
% A U T H O R       D E T A I L S 

\noindent G. Arunkumar\\
\textsc{Department of Mathematics, IIT Madras, Chennai 600036.}\\
\textit{Email address: }\texttt{\href{mailto:garunkumar@iitm.ac.in}{garunkumar@iitm.ac.in}}\\

\medskip 

\noindent Narayanan Narayanan\\
\textsc{Department of Mathematics, IIT Madras, Chennai 600036.}\\
\emailad{naru@iitm.ac.in}\\
%\textit{Email address: }\texttt{naru@iitm.ac.in}\\

\medskip 

\noindent Raghavendra Rao B. V.\\
\textsc{Department of Computer Science and  Engineering,  IIT Madras, Chennai 600036.}\\
\emailad{bvrr@iitm.ac.in}\\

\medskip

\noindent Sagar S. Sawant\\
\textsc{Department of Mathematics,  IIT Madras, Chennai 600036.}\\
\emailad{sagar@smail.iitm.ac.in}\\

\end{document}